\documentclass[11pt]{amsart}

\usepackage[utf8]{inputenc}
\usepackage[english]{babel}
\usepackage[a4paper,twoside,top=1.2in, bottom=1.1in, left=0.8in, right=0.8in]{geometry}

\usepackage{graphicx}
\usepackage{caption} 
\usepackage{amsmath}
\usepackage{amsthm}
\usepackage{amssymb}
\usepackage{esint} 
\usepackage{mathrsfs} 
\usepackage{xcolor}
\usepackage{array}
\usepackage{hhline}
\usepackage{enumitem} 
\usepackage{comment} 
\usepackage[toc,page]{appendix} 
\usepackage{xparse} 
\usepackage{mathtools}
\usepackage{fancyhdr} 
\usepackage{ifthen} 
\usepackage{forloop} 
\usepackage{xstring}
\usepackage{emptypage} 
\usepackage{setspace}
\usepackage[initials,alphabetic]{amsrefs} 

\usepackage{tikz}
\usetikzlibrary{arrows,shapes,patterns,calc,fadings,decorations.pathreplacing,decorations.markings,decorations.pathmorphing,backgrounds}

\usepackage{hyperref} 
\hypersetup{
	colorlinks = true,
	linkcolor = {blue},
	urlcolor = {red},
	citecolor = {blue}
}

\usepackage[nameinlink,capitalise]{cleveref} 

\newcounter{results}[section] 

\theoremstyle{plain}
\newtheorem{theorem}[results]{Theorem}
\newtheorem{lemma}[results]{Lemma}
\newtheorem{proposition}[results]{Proposition}
\newtheorem{corollary}[results]{Corollary}

\newtheorem*{theorem*}{Theorem}
\newtheorem*{lemma*}{Lemma}
\newtheorem*{proposition*}{Proposition}
\newtheorem*{corollary*}{Corollary}
\newtheorem*{exercise*}{Exercise}
\newtheorem*{fact*}{Fact}

\theoremstyle{remark}
\newtheorem{remark}[results]{Remark}

\newtheorem*{remark*}{Remark}
\newtheorem*{question*}{Question}

\theoremstyle{definition}

\newtheorem*{definition*}{Definition}
\newtheorem*{example*}{Example}

\numberwithin{equation}{section}

\crefname{figure}{Figure}{Figures}

\ifdefined\comma 
        \renewcommand{\comma}{\ensuremath{\, \text{, }}}
\else 
        \newcommand{\comma}{\ensuremath{\, \text{, }}}
\fi

\newcommand{\pheq}{\ensuremath{\hphantom{{}={}}}}

\newcommand{\R}{\ensuremath{\mathbb R}}

\DeclarePairedDelimiter\abs{\lvert}{\rvert} 
\newcommand{\scal}[2]{\ensuremath{\langle #1 , #2 \rangle}} 

\newcommand{\st}{\ensuremath{\ :\ }} 
\newcommand{\eqdef}{\ensuremath{\coloneqq}} 

\DeclareMathOperator{\tr}{tr}

\newcommand{\de}{\ensuremath{\, d}} 

\let\div\undefined
\newcommand{\div}{\ensuremath{\mathrm{div}}} 
\newcommand{\grad}{\ensuremath{\nabla}} 
\newcommand{\Hess}{\ensuremath{\operatorname{Hess}}} 

\newcommand{\vf}{\ensuremath{\mathfrak X}} 
\DeclareMathOperator{\vol}{vol} 

\newcommand{\gb}{\ensuremath{g}} 
\newcommand{\gc}{\ensuremath{{\tilde {g}}}} 

\colorlet{myGray}{gray}
\colorlet{myBlue}{blue}
\colorlet{myBlack}{black}
\colorlet{myBackground}{gray!10}

\title[On the stability of minimal submanifolds in conformal spheres]{On the stability of minimal submanifolds\\ in conformal spheres}
\author{Giada Franz and Federico Trinca}
\newcommand\printaddress{{
\setlength{\parindent}{17pt}
\footnotesize

\bigskip
\par 
{\scshape \noindent Giada Franz}
\newline MIT, Department of Mathematics, 77 Massachusetts Avenue, Cambridge, MA 02139, USA
\newline
\textit{E-mail address:} \texttt{gfranz@mit.edu}
\newline
\par
{\scshape \noindent Federico Trinca}
\newline Mathematical Institute, University of Oxford, Woodstock Road, Oxford, OX2 6GG, United Kingdom
\newline
\textit{E-mail address:} \texttt{federico.trinca@maths.ox.ac.uk}
\par
}} 

\begin{document}
\maketitle

\thispagestyle{empty}

\begin{abstract}
Given an $n$-dimensional Riemannian sphere conformal to the round one and $\delta$-pinched, we show that it does not contain any closed stable minimal submanifold of dimension $2\le k\le n-\delta^{-1}$. 
\end{abstract}

\section{Introduction}

Given a compact Riemannian manifold $M^n$ with dimension $n\ge 3$, a closed $k$-dimensional minimal submanifold $\Sigma^k$ in $M$ is defined as a critical point of the $k$-volume.
Hence, it is natural to look at what happens to the second derivative of the $k$-volume at $\Sigma$.
Of particular interest is the case when this second derivative is nonnegative, if so we say that $\Sigma$ is \emph{stable}.
Note that, if a submanifold is a minimizer of the $k$-volume, then it is a stable minimal submanifold.

Lawson and Simons in \cite{LawsonSimons1973}*{pp. 438} conjectured that there are no stable minimal submanifolds in any compact, simply-connected, Riemannian manifold $M$ that is $\frac 14$-pinched.
Here, by $\delta$-pinched for some $\delta > 0$ we mean that at each point of $M$ the sectional curvatures are positive and the ratio between the smallest and the largest sectional curvatures of $M$, at that point, is strictly bigger than~$\delta$.
Observe that, thanks to the sphere theorem \cite{BrendleSchoen2009}*{Theorem 1}, a $\frac 14$-pinched Riemannian manifold $M$ is diffeomorphic to the round unit sphere. 
We remark that this pinching condition is sharp both for the Lawson--Simons conjecture and the sphere theorem. We will come back to this later in the introduction (see \cref{sec:heuristics})

The purpose of the present paper is to investigate this problem in the case where the ambient manifold is conformal to the round unit sphere. In particular, we obtain the following result.

\begin{theorem} \label{thm:main}
Let $(S^n,\gc)$ be a $n$-dimensional Riemannian manifold that is conformal to the round unit sphere and that is $\delta$-pinched for some $\delta>0$. Then $(S^n,\gc)$ does not contain any closed stable minimal $k$-submanifold for all $2\le k\le n-\delta^{-1}$.
\end{theorem}

\begin{remark}
Observe that \cref{thm:main} confirms Lawson--Simons conjecture in conformal spheres $S^n$ for all $k$-submanifolds with dimension $2\le k\le n-4$.
\end{remark}

\subsection{Previous results}

The original motivation behind Lawson--Simons conjecture arises from the fact that there are no closed stable minimal submanifolds in the round $n$-dimensional sphere. This was proven by Simons in \cite{Simons1968}*{Theorem 5.1.1}, and generalized to the nonexistence of stable stationary varifolds in Corollary 1 of the aforementioned paper \cite{LawsonSimons1973}.
The idea of the proof is to show that, given a minimal $k$-submanifold $\Sigma^k$, the trace of the second derivative of the $k$-volume on the space of variations generated by constant vector fields is negative. Hence, there is at least one constant vector field that generates a variation along which the $k$-volume of $\Sigma$ decreases, which implies that $\Sigma$ is not stable.

The Lawson--Simons conjecture in its full generality is still open, but there have been several partial results in its support. 
\begin{itemize}
\item \emph{Minimal two-spheres.} Let us assume that $\Sigma^2$ is a minimal surface homeomorphic to a two-sphere in a compact, simply-connected, $\frac 14$-pinched Riemannian manifold $M^n$.

In \cite{Aminov1976}*{Theorem 1}, Aminov proved that, under these assumptions, $\Sigma$ is unstable. The proof consists in constructing two orthogonal variations, determined by a system of differential equations \cite{Aminov1976}*{(10)}, with negative average of the second derivative of the area. 

Aminov argument was then refined by Micallef and Moore \cite{MicallefMoore1988}*{Theorem 1}, who obtained a positive lower bound on the number of negative directions of the second derivative of the area (the \emph{Morse index}) at $\Sigma$ as above. More precisely, they proved that the Morse index of $\Sigma$ in $M$ is strictly bigger than $(n-3)/2$, thus the lower bound only depends on the dimension of the ambient manifold.  
\item \emph{Perturbations of the round metric.} If $S^n$ is endowed with a metric which is $C^2$-close to the round one, then, Howard and Wei in \cite{HowardWei2015}*{Theorem 1} proved that there are no stable minimal submanifolds in it. Note that \cite{HowardWei2015} circulated as a preprint since 1983.

Quantitative versions of this result were proven by Okayasu and Howard. More precisely, given $g\ge 1$, Okayasu in \cite{Okayasu1989}*{Theorem A} showed that there exists $\varepsilon'=\varepsilon'(n,g)>0$ such that, if $S^n$ is $(1-\varepsilon')$-pinched, then it does not contain any closed stable minimal surface of genus $g$. 
Moreover, given $k\ge 1$, Howard in \cite{Howard1985}*{Theorem~1} proved the existence of $\varepsilon''=\varepsilon''(n,k)>0$ such that, if $S^n$ is $(1-\varepsilon'')$-pinched, then there are no closed stable minimal submanifolds of dimension $k$. 
Unfortunately, it holds that $\lim_{g\to\infty} \varepsilon'(n,g) = 0$ for all $n\ge 3$, and that $\lim_{n\to\infty}\max_{1\le k < n} \varepsilon''(n,k) = 0$.

\item \emph{Ambient manifold immersed in $\R^{n+1}$.} The last case that was considered is when the ambient $n$-dimensional topological sphere can be isometrically immersed in $\R^{n+1}$. Under this additional condition, the Lawson--Simons conjecture was proven by Shen and Xu in \cite{ShenXu2000}*{Theorem 2}. Further results with better pinching conditions were also proven (see \cites{ShenHe2001,HuWei2003}).
\end{itemize}

\begin{remark}
In \cite[Theorem 1]{ShenXu2000}, Shen and Xu claimed that there are no stable minimal submanifolds in a complete simply-connected $0.77$-pinched Riemannian manifold. However, we are actually unable to check the validity of the second equality in (3.6).
\end{remark}

What we do in \cref{thm:main} is to consider a setting transversal to all the previous ones, namely the case where the ambient manifold is conformal to the round unit sphere and satisfies some pinching condition.
Note that such an ambient manifold of dimension $n$ can be isometrically immersed in $\R^{n+1}$ if and only if it satisfies some local rigidity properties by \cite{NishikawaMaeda1974}*{Theorem 4} (see also Remark 2 after the theorem therein). Hence, our result is not a consequence of \cite[Theorem 2]{ShenXu2000}.

\subsection{Heuristics and further developments} \label{sec:heuristics}

Apart from the aforementioned works concerning the Lawson--Simons conjecture, there are other results on the stability of minimal submanifolds that are worth noticing. First of all, it is immediate to see that, in a space of positive sectional curvature, a minimal submanifold with a nontrivial $\nabla^\perp$-parallel section of the normal bundle is unstable. In particular, this section always exists on hypersurfaces with trivial normal bundle (see the \emph{stability inequality}). As a consequence, the Lawson--Simons conjecture trivially holds for $n=2$, and more generally for two-sided hypersurfaces.

\begin{remark}
The case of minimal hypersurfaces is actually much more studied and understood than the higher codimensional one. 
In particular, there are many results relating the Morse index and the topology of a minimal hypersurface (see e.g. \cite{Ros2006,Savo2010,AmbrozioCarlottoSharp2018Index, Souam2015}). Note that \cite{Souam2015} deals with the conformal case as in the present work.
\end{remark}

Another important special case is the geodesics one. Indeed, the celebrated theorem of Synge (see \cite[Theorems 21 and 26]{Petersen2016}) says that geodesics in even dimensional compact simply-connected Riemannian manifolds of positive sectional curvature are unstable. 
However, this cannot be true in odd dimensions. Indeed, Ziller \cite[Example 1]{Ziller1977} showed that there is a stable closed geodesic in each Berger $3$-sphere that is $\delta$-pinched with $\delta\in(0,\frac 15)$. 

As mentioned above, the pinching condition of the Lawson--Simons conjecture is sharp. Indeed, the complex projective space is a compact simply connected Riemannian manifold with sectional curvature between $1/4$ and $1$, included, that admits a large class of minimizers for the volume functional: the complex submanifolds. Being homologically volume minimizers, complex submanifolds are intimately related to the homology of the ambient manifold. In light of this discussion, Howard and Wei \cite[Conjecture B]{HowardWei2015} conjectured that homological spheres with positive sectional curvature admit no stable minimal submanifolds. 
However, this conjecture cannot hold in this great generality, as showed by Ziller's example described above. More in general, Torralbo--Urbano in \cite{TorralboUrbano2022} proved that any odd-dimensional Berger sphere that is $\delta$-pinched with $\delta\in(0,\frac{1}{4k+1})$ contains $k$-dimensional stable minimal submanifolds. Observe that this is compatible with the pinching obtained by Ziller for $k=1$.

\subsection{Idea of the proof and structure of the paper}

The idea behind \cref{thm:main} is inspired by Simons' proof, adequately adapted to the conformal case.  Indeed, we consider the constant vector fields used by Simons, and we rescale them so that they are orthonormal under the conformal change of metric. Given a $k$-submanifold that is minimal with respect to the conformal metric, we take the trace of the second derivative of the conformal $k$-volume over the space of variations generated by these vector fields, and we express such a trace in terms of objects relative to the round metric. 

Now, we discuss the novel key ideas of this paper.
First, we observe that Simons' estimate for the operator associated to the second variation over the space $\mathcal{V}$ of constant vector fields holds for \emph{nonminimal} submanifolds when we trace over $\mathcal V$ (see \cref{thm:LawsonSimonsClosed}). This way, we are able to obtain an expression for the aforementioned trace depending only on the curvature of the conformal sphere and the norm of the mean curvature of the submanifold with respect to the round metric (see \cref{thm:TraceSecondVariationCC}). We conclude estimating the aforesaid norm with terms involving the curvature of the conformal sphere.
This final step consists in applying the divergence theorem to a carefully chosen vector field to recover terms involving the curvature. 
As pointed out by the anonymous referee, who we would like to thank, this also follows from integrating the Gauss equation for the scalar curvature of $\Sigma$, combined with the standard formulas for the scalar curvature and the traceless second fundamental form under a conformal change of metric.

The paper is structured as follows. In \cref{sec:Prelims}, we review standard results from conformal geometry. In particular, we recall how some usual Riemannian geometric quantities transform under a conformal change of metric. We use these results in \cref{sec:2VUnderConformal} to compute the operator associated to the second variation of the conformal $k$-volume in terms of the same operator with respect to the round metric. Finally, \cref{sec:ProofMainThm} is devoted to the proof of \cref{thm:main}. Specifically, we trace over the space of rescaled constant vector fields, and we estimate the norm of the mean curvature of the submanifold with respect to the round metric.

\subsection*{Acknowledgements}
This project originated from the Master thesis of the second author under the supervision of Alessandro Carlotto, whom we would like to thank for suggesting the problem and for all the valuable feedback.
We would also like to thank Robert Bryant for pointing out the reference \cite{NishikawaMaeda1974}.
This project has received funding from the European Research Council (ERC) under the European Union’s Horizon 2020 research and innovation programme (grant agreement No. 947923).

\section{Setting and notation} \label{sec:notation}

Let $(M^n,g)$ be an $n$-dimensional compact Riemannian manifold, and assume that $\tilde g\eqdef e^{2f}g$ is a Riemannian metric on $M$ conformal to $g$ with conformal factor $e^{2f}$, for some function $f\in C^\infty(M)$.
Moreover, let $\Sigma^k\subset M$ be a closed $k$-dimensional submanifold of $M$.
Note that we can view $\Sigma^k$ as a Riemannian submanifold both of $(M,\gb)$ and of $(M,\gc)$. Obviously, the metric induced by $\gc$ is conformal to the one induced by $\gb$ with conformal factor $e^{2f}$.

In this setting we adopt the following notation.
\begin{itemize}
\item $\vf(M)$ is the set of vector fields on $M$.
\item $X^\perp\in\Gamma(N\Sigma)$ is the normal component to $\Sigma$ of a vector field $X\in\vf(M)$, where $\Gamma(N\Sigma)$ denotes the sections of the normal bundle of $\Sigma$. Observe that the normal component with respect to $g$ is the same as the normal component with respect to the conformal metric $\tilde g$.
\item $X^\top = X-X^\perp$ is the tangent component to $\Sigma$ of a vector field $X\in\vf(M)$. Again this does not depend on the choice of metric in a conformal class.
\item $\{E_\alpha\}_{\alpha=1}^n$ will often denote a local orthonormal basis of $M$ with respect to $g$. When this basis is defined around a point $p\in\Sigma$, then at $p$ it will be chosen such that $\{E_i\}_{i=1}^k$ is an orthonormal basis of $\Sigma$ and $\{E_r\}_{r=k+1}^n$ is an orthonormal basis of the normal to $\Sigma$.
\item $\nabla$, $\tilde \nabla$ are the Levi-Civita connections on $M$ with respect to $g$ and $\tilde g$, respectively.
\item $\nabla^\perp$ is the Levi-Civita connection on the normal bundle of $\Sigma$. In particular, $\nabla^\perp_XV = (\nabla_XV)^\perp$ for all $X\in\vf(\Sigma)$ and $V\in\Gamma(N\Sigma)$.
\item $R_M$, $\tilde R_M$ are the Riemann curvature tensors on $M$ with respect to $g$ and $\tilde g$, respectively. Here, we use the convention
\[
R_M(X,Y)Z = \nabla_Y\nabla_XZ - \nabla_X\nabla_YZ + \nabla_{[X,Y]}Z
\]
for all $X,Y,Z\in\vf(M)$.
\item If $\pi\subset T_p M$ is a two-plane of $M$, then, $K_M(\pi)$, $\tilde K_M(\pi)$ are the sectional curvatures of $\pi$ on $M$ with respect to $g$ and $\tilde g$, respectively. If $\pi$ is spanned by $X,Y\in T_pM$, then, $K_M(\pi)=K_M(X,Y)$ and $\tilde{K}_M(\pi)=\tilde{K}_M (X,Y)$.
\item We say that $(M,\tilde g)$ is \emph{$\delta$-pinched} for some $\delta>0$ if at each point $p\in M$ and for every pair of two-planes $\pi_1,\pi_2\subset T_p M$, we have  $0<\delta \tilde{K}_M(\pi_1)< \tilde{K}_M(\pi_2)$. In particular, for all $p\in M$, it holds
\[
0<\delta \max_{\pi\subset T_pM}\tilde{K}_M(\pi)< \min_{\pi\subset T_pM}\tilde{K}_M(\pi).
\]
\item $A$, $\tilde A$ are the second fundamental forms of $\Sigma$ on $M$ with respect to $g$, $\tilde g$, respectively. Moreover, taking the traces of $A$ and $\tilde A$ over $\Sigma$, we obtain the mean curvature vectors $H$, $\tilde H$ of $\Sigma$ on $M$ with respect to the two metrics $g$, $\tilde g$.
\end{itemize}

Observe that, in \cref{sec:ProofMainThm}, we specialize to the case where $(M^n,g)$ is the unit sphere $S^n$ with the round metric and $\Sigma$ is a minimal submanifold with respect to $\tilde g$.
Note that, under these assumptions, we have that $\tilde H =0$, but the mean curvature $H$ with respect to $g$ does not need to vanish.

\section{Preliminaries in conformal geometry}\label{sec:Prelims}

Let us assume to be in the setting described in the previous section, namely an $n$-dimensional Riemannian manifold $M^n$ with conformal metrics $\gc = e^{2f}\gb$.
The following well-known results relate the geometry of the conformal metric $\gc$ to the original one $\gb$.

\begin{lemma}\label{lem:ConformalChangeBasis}
In the setting above, let $\{E_\alpha\}_{\alpha=1}^n$ be an orthonormal frame of $(M,\gb)$ defined on an open subset $U\subset M$. Then, $\{\tilde{E}_\alpha\}_{\alpha=1}^n$ is an orthonormal frame of $(M,\gc)$ defined on $U\subset M$, where $\tilde{E}_\alpha\eqdef e^{-f}E_\alpha$ for all $\alpha\in\{1,\ldots,n\}$.
\end{lemma}

\begin{proposition}[{cf. \cite[Theorem~1.159]{Besse1987}}] \label{prop:UnderConformalChange} 
In the setting above, we have that
\begin{enumerate} [label={\normalfont(\roman*)}]
    \item\label{ucc:conn} the Levi-Civita connection of $(M,\gb)$ is related to the Levi-Civita connection of $(M,\gc)$ by 
    $$\tilde{\nabla}_{X} Y=\nabla_X Y +X(f) Y+Y(f) X-\gb(X,Y) \nabla f$$
    for all vector fields $X,Y\in \vf(M)$;
    \item\label{ucc:curv} the Riemann curvature tensor of $(M,\gb)$ is related to the Riemann curvature tensor of $(M,\gc)$ by
        \begin{align*}
            \tilde{R}_M&(X,Y) Z= R_M(X,Y)Z +X(f)Z(f)Y-Y(f)Z(f)X + {} \\
              &\pheq{}-X(f) \gb(Y,Z) \nabla f+Y(f) \gb(X,Z) \nabla f-\gb(X,Z) \nabla_Y \nabla f+ \gb(Y,Z) \nabla_X \nabla f + {}\\
              &\pheq {}-\gb(X,Z)\abs{\nabla f}^2_{\gb}Y+ \gb(Y,Z)\abs{\nabla f}^2_{\gb} X-\Hess f (X,Z) Y+\Hess f(Y,Z) X
         \end{align*}
         for all $X,Y,Z\in\vf(M)$;
    \item\label{ucc:sect} the sectional curvature of $(M,\gb)$ is related to the sectional curvature of $(M,\gc)$ by
        $$e^{2f} \tilde{K}_M (X,Y)=K_M (X,Y)+X(f)^2+Y(f)^2-\abs{\nabla f}^2_{g} -\Hess f(X,X) -\Hess f (Y,Y)$$
         for all $X,Y\in \vf(M)$ that are orthonormal with respect to $\gb$;
    \item the volume element $d\gb$ of $(M,\gb)$ is related to the volume element $d\gc$ of $(M,\gc)$ by 
         $$d\gc=e^{nf}d\gb.$$
\end{enumerate}
\end{proposition}

The following lemma shows how the mean curvatures with respect to $g$ and $\tilde g$ of a $k$-dimensional submanifold of $M$ are related.
\begin{lemma} \label{lem:MeanCurvatureCC}
Let $\Sigma^k$ be a submanifold of $M$. Then, the mean curvature $H$ of $\Sigma$ with respect to $\gb$ is related to the mean curvature $\tilde{H}$ of $\Sigma$ with respect to $\gc$, by
\begin{align*}
    H=e^{2f} \tilde{H}+k\nabla^\perp f.
\end{align*}
\end{lemma}
\begin{proof}
Given a local orthonormal frame $\{E_{\alpha}\}_{\alpha=1}^n$ of $(M,\gb)$ such that $\{E_i\}_{i=1}^k$ are tangent to $\Sigma$, we can use \cref{lem:ConformalChangeBasis} to obtain a local $\gc$-orthonormal frame $\{\tilde{E}_{\alpha}\}_{\alpha=1}^n$, which will also have the first $k$ vectors tangent to $\Sigma$. We deduce 
\begin{align*}
    H &=\sum_{i=1}^k \left(\nabla_{E_i} E_i\right)^{\perp}=\sum_{i=1}^k e^f \left( \nabla_{\tilde{E}_i} (e^f \tilde{E}_i)\right)^{\perp}=\sum_{i=1}^k e^f \tilde{E}_i (e^f) \left(\tilde{E}_i\right)^{\perp} +\sum_{i=1}^k e^{2f} \left(\nabla_{\tilde{E}_i} \tilde{E}_i\right)^{\perp}\\
    &=e^{2f} \sum_{i=1}^k (\tilde{\nabla}_{\tilde{E}_i} \tilde{E}_i)^{\perp}+e^{2f} \sum_{i=1}^k \gb(\tilde{E}_i,\tilde{E}_i) (\nabla f)^{\perp}=e^{2f} \tilde{H}+k \nabla^\perp f,
\end{align*}
where we used \cref{prop:UnderConformalChange}\ref{ucc:conn}, the properties of the connection and the fact that $\perp$ does not depend on conformal changes.
\end{proof}

\section{Second variation of the \texorpdfstring{$k$}{k}-volume after a conformal change of metric}\label{sec:2VUnderConformal}

Let $\Sigma^k$ be a $k$-dimensional submanifold of $M$.
Consider the quadratic operator on the normal bundle of $\Sigma$ defined, for all $V\in\Gamma(N \Sigma)$, as follows
\begin{align*}
    \tilde {Q}^{\Sigma}(V,V)\eqdef \int_{\Sigma} \tilde{q}^{\Sigma} (V,V)\de \gc ,
\end{align*}
where $\tilde q^\Sigma$ is, at each given point, the quadratic operator given by
\[
\tilde{q}^{\Sigma} (V,V)\eqdef \lvert\tilde{\nabla}_{\Sigma}^\perp V\rvert^2_\gc- \widetilde \tr_{\Sigma} (\tilde{R}_M(V,\cdot)V)-
    \lvert \gc(\tilde{A}(\cdot,\cdot),V) \rvert^2.
\]
Here, $ \widetilde\tr_{\Sigma}$ denotes the trace on $\Sigma$ with respect to the metric $\gc$, namely
\[
\widetilde \tr_{\Sigma} (\tilde{R}_M(V,\cdot)V ) = \sum_{i=1}^k \gc(\tilde{R}_M(V,\tilde E_i)V,\tilde E_i),
\]
where $\{\tilde E_i\}_{i=1}^k$ is an orthonormal basis of $\Sigma$ with respect to $\gc$.
Observe that, if $\Sigma$ is minimal with respect to $\gc$, $\tilde {Q}^{\Sigma}$ is the associated stability (or Jacobi) operator (see e.g. \cite[Chapter 1, \S8]{ColdingMinicozzi2011}).
Therefore, $\Sigma$ is unstable if and only if there exists $V\in \Gamma(N\Sigma)$ such that $\tilde Q^\Sigma(V,V) < 0$.

Now assume, as above, that $\gc = e^{2f}\gb$ is conformal to a reference metric $\gb$ on $M$. 
The goal of this section is to express $\tilde Q^\Sigma(V,V)$ for every $V\in\Gamma(N\Sigma)$ in terms of the corresponding operator $Q^\Sigma(V,V)$ with respect to the metric $\gb$, namely
\begin{align*}
    Q^{\Sigma}(V,V)=\int_{\Sigma} q^{\Sigma} (V,V) \de \gb= \int_{\Sigma} \lvert\nabla_{\Sigma}^\perp V\rvert^2_{\gb}- \tr_\Sigma( R_M(V,\cdot)V)-
    \lvert \gb(A(\cdot,\cdot),V) \rvert^2 \de \gb.
\end{align*}

For the rest of the section, we assume that $\{E_\alpha\}_{\alpha=1}^n$ is a local orthonormal frame with respect to $\gb$ such that the first $k$ terms $\{E_i\}_{i=1}^k$ form an orthonormal basis of $\Sigma$.
\begin{lemma}
For every $V\in\Gamma(N\Sigma)$, it holds
\begin{align*}
    \lvert\tilde{\nabla}_{\Sigma}^\perp V\rvert^2_\gc=\lvert\nabla_{\Sigma}^\perp V\rvert^2_{\gb}+  (\nabla^\top f)(\abs{V}^2_\gb)+\abs{\grad^\top f}^2_\gb \abs{V}^2_\gb .
\end{align*}
\end{lemma}
\begin{proof}
Observe that \cref{prop:UnderConformalChange}\ref{ucc:conn} implies
\begin{align*}
    \tilde{\nabla}^\perp_{\tilde{E}_i} V=\left(\nabla_{\tilde{E}_i} V+\tilde{E}_i(f) V\right)^\perp=\nabla^\perp_{\tilde{E}_i} V+\tilde{E}_i(f) V,
\end{align*}
for all $i\in\{1,\ldots,k\}$. Taking the inner product, we compute
\begin{align*}
    \gc(\tilde{\nabla}^\perp_{\tilde{E}_i} V, \tilde{\nabla}^\perp_{\tilde{E}_i} V)&=\gc(\nabla^\perp_{\tilde{E}_i} V,\nabla^\perp_{\tilde{E}_i} V)+\tilde{E}_i(f)^2 \gc(V,V)+2\tilde{E}_i(f)\gc(V,\nabla^\perp_{\tilde{E}_i} V)\\
    &= \gb(\nabla_{E_i}^\perp V,\nabla_{E_i}^\perp V) + E_i(f)^2\gb(V,V) + 2E_i(f) \gb(V,\nabla_{E_i}^\perp V),
\end{align*}
which gives the lemma after summing over $i\in\{1,\ldots,k\}$ and using that
\[
2\sum_{i=1}^k E_i(f) \gb(\nabla_{E_i}^\perp V,V) = 2\gb(\nabla^\perp_{\nabla^\top f}V,V) = 2\gb(\nabla_{\nabla^\top f}V,V) =(\nabla^\top f)(\abs{V}^2_\gb). \qedhere
\]
\end{proof}

\begin{lemma}
For every $V\in\Gamma(N\Sigma)$, it holds
\begin{align*}
    \widetilde \tr_{\Sigma} (\tilde{R}_M(V,\cdot)V)&=\tr_\Sigma(R_M(V,\cdot)V)+kV(f)^2-k\Hess f(V,V)-k\abs{\nabla f}^2_{\gb} \abs{V}^2_\gb +{}\\
    &\pheq -\div_\Sigma(\nabla f) \abs{V}^2_\gb +\abs{\grad^\top f}^2_\gb \abs{V}^2_\gb.
\end{align*}
\end{lemma}
\begin{proof}
Using \cref{prop:UnderConformalChange}\ref{ucc:curv} and the fact that $\Hess f(X,Y)=\gb(\nabla_Y \nabla f,X)$, we have
\begin{align*}
    \gc(\tilde{R}_M(V,\tilde{E}_i)V,\tilde{E}_i)&=\gc(R_M(V,\tilde{E}_i)V,\tilde{E}_i)+V(f)^2\gc(\tilde{E}_i,\tilde{E}_i)+\tilde{E}_i (f)\gb(V,V)\gc(\nabla f,\tilde{E}_i)+ {}\\
    &\pheq -\gb(V,V)\gc(\nabla_{\tilde{E}_i} \nabla f,\tilde{E}_i)-\abs{\nabla f}^2_{\gb}\gb(V,V) \gc(\tilde{E}_i,\tilde{E}_i)-\Hess f(V,V)\gc(\tilde{E}_i,\tilde{E}_i)\\
    &= \gb(R_M(V,E_i)V,E_i)+V(f)^2+E_i(f)^2\gb(V,V)+{}\\
    &\pheq -\gb(V,V)\Hess f(E_i,E_i)-\abs{\nabla f}^2_{\gb} \gb(V,V)-\Hess f(V,V).
\end{align*}
Summing over $i\in\{1,\ldots,k\}$ we conclude. 
\end{proof}

\begin{lemma}
For every $V\in\Gamma(N\Sigma)$, it holds
\begin{align*}
    \abs{\gc(\tilde{A}(\cdot,\cdot),V)}^2=\abs{\gb ({A}(\cdot,\cdot),V)}^2+kV(f)^2-2V(f)\gb(H,V).
\end{align*}
\end{lemma}
\begin{proof}
By \cref{prop:UnderConformalChange}\ref{ucc:conn}, note that
\begin{align*}
    \tilde{A}(\tilde{E}_i,\tilde{E}_j)=(\tilde\nabla_{\tilde E_i}\tilde E_j)^\perp=A(\tilde{E}_i,\tilde{E}_j)-e^{-2f}\delta_{ij}(\nabla f)^\perp  ,
\end{align*}
    which implies
   \begin{align*}
       \gc(\tilde{A}(\tilde{E}_i,\tilde{E}_j),V)=\gc(A(\tilde{E}_i,\tilde{E}_j)-e^{-2f}\delta_{ij}(\nabla f)^\perp  ,V) = \gb(A(E_i,E_j),V)-\delta_{ij} V(f).
   \end{align*}
   Taking the absolute value squared, we have
   \begin{align*}
       \abs{\gc(\tilde{A}(\tilde{E}_i,\tilde{E}_j),V)}^2=\abs{\gb(A(E_i,E_j),V)}^2+\delta_{ij} V(f)^2-2\delta_{ij} V(f)\gb(A({E}_i,{E}_j),V).
   \end{align*}
   We can conclude by summing over $i,j\in\{1,\ldots,k\}$ and observing that
   \begin{align*}
       -2\sum_{i,j=1}^k\delta_{ij} V(f)\gb(A({E}_i,{E}_j),V)&=-2 V(f) g\left(\sum_{i=1}^k {A}(E_i,E_i),V\right) = -2V(f) \gb(H,V). \qedhere
   \end{align*}
\end{proof}

\begin{proposition} \label{prop:SmallQTransf}
Assume that $\Sigma$ is minimal with respect to $\gc$, then for every $V\in\Gamma(N\Sigma)$ we have
\begin{align*}
    \tilde{q}^{\Sigma}(V,V)&= q^{\Sigma} (V,V) + (\nabla^\top f)(\abs{V}^2_\gb)+k\Hess f(V,V) + k\abs{\nabla f}^2_{\gb} \abs{V}_\gb^2+\div_\Sigma(\nabla f) \abs{V}_\gb^2 .
\end{align*}
\end{proposition}
\begin{proof}
The statement follows directly from the previous lemmas and 
\[
\gb(H,V) = \gc(\tilde H,V) + k \gb(\nabla f, V) = kV(f),
\]
which is a consequence of \cref{lem:MeanCurvatureCC} and $\tilde H = 0$.
\end{proof}
\begin{corollary}\label{cor:SecondVariationCC}
Assume that $\Sigma$ is minimal with respect to $\gc$, then for every $V\in\Gamma(N\Sigma)$ it holds
\begin{align*}
    \tilde{q}^{\Sigma}(\tilde{V},\tilde{V})&= q^{\Sigma} (V,V) e^{-2f} -\abs{\nabla^\top f}^2_\gb \abs{V}^2_\gb e^{-2f}+k\Hess f(V,V)e^{-2f} +{}\\
    &\pheq + k\abs{\nabla f}^2_{\gb} \abs{V}^2_\gb e^{-2f}+\div_\Sigma(\nabla f) \abs{V}^2_\gb e^{-2f},
\end{align*}
where $\tilde{V}=e^{-f} V$.

\end{corollary}

\begin{proof}
The first observation is that all the terms of $q^\Sigma$ are tensorial apart from $\abs{\nabla^\perp_{E_i} \tilde{V}}_{\gb}^2$, for which we have
\begin{align*}
    \nabla^\perp_{E_i} \tilde{V}=E_i(e^{-f}) V+e^{-f}\nabla^\perp_{E_i}V=-e^{-f}E_i(f)V+e^{-f}\nabla^\perp_{E_i}V,
\end{align*}
and, hence
\begin{align*}
    \abs{\nabla^\perp_{{E_i}} \tilde{V}}_{\gb}^2&=e^{-2f}{E}_i(f)^2\gb(V,V)-2e^{-2f}E_i(f)\gb(\nabla^\perp_{{E}_i}V, V)+e^{-2f}\gb(\nabla^\perp_{{E}_i} V,\nabla^\perp_{{E}_i} V).
\end{align*}
Summing over $i\in\{1,\ldots,k\}$, we get
\[
\abs{\nabla^\perp_{\Sigma} \tilde{V}}_{\gb}^2 =  
\abs{\nabla^\perp_{\Sigma} {V}}_{\gb}^2 e^{-2f} + \abs{\nabla^\top f}^2_\gb \abs{V}^2_\gb e^{-2f} - (\nabla^\top f)(\abs{V}^2_\gb) e^{-2f}.
\]
The only other term in the equation of \cref{prop:SmallQTransf} that is not tensorial is $(\nabla^\top f)(\abs{\tilde V}^2_\gb)$, for which we have
\[
(\nabla^\top f)(e^{-2f}\abs{V}^2_\gb) = -2 \abs{\nabla^\top f}^2_\gb \abs{V}^2_ge^{-2f} + (\nabla^\top f)(\abs{V}^2_\gb) e^{-2f}.
\]
Combining these calculations with \cref{prop:SmallQTransf}, we conclude.
\end{proof}

\section{Proof of the main result}\label{sec:ProofMainThm}
In this section, we prove our main theorem. We take as base manifold $(M^n,g)$ the round unit sphere $(S^n,\gb)$ and we consider a function $f\in C^{\infty}(M)$, which induces a fixed conformal metric $\gc = e^{2f}\gb$ on~$S^n$. 

Let $\mathcal{F}\eqdef\left\{h\big|_{S^n}: h\in \hom(\R^{n+1},\R)\right\}$ be the family of linear maps from $\R^{n+1}$ to $\R$ restricted to $S^n$, and let $\mathcal{V}\eqdef\left\{\nabla h: h\in \mathcal{F}\right\}\subset \vf(S^n)$ be the projection to $S^n$ of the constant vector fields of $\R^{n+1}$. The natural isomorphism $\R^{n+1}\cong\mathcal{V}$, which associates to any vector $v$ of $\R^{n+1}$ the gradient of the function $x\mapsto \langle v,x\rangle$ on $S^n$, induces a natural inner product on $\mathcal{V}$. 

Simons in \cite[Lemma 5.1.4]{Simons1968} observes that each element of $\mathcal{V}$ is a negative direction of the $k$-volume for every $k$-dimensional closed minimal submanifold $\Sigma^k$ of the round unit sphere. In particular, the author proves that
\[
Q^\Sigma(V^\perp,V^\perp) = -k \int_\Sigma \abs{V^\perp}^2_\gb \de \gb,
\]
for all $V\in\mathcal{V}$.

In what follows, we consider \emph{any} closed submanifold $\Sigma^k$ of $S^n$, without the assumption of minimality. The previous equation for $Q^\Sigma(V^\perp,V^\perp)$ does not hold for all $V\in\mathcal{V}$. However, we surprisingly recover the same result when we trace over $\mathcal{V}$.

\begin{theorem}\label{thm:LawsonSimonsClosed}
Let $\Sigma^k$ be a $k$-dimensional closed submanifold of $(S^n,\gb)$.
Then, if we take the trace of $q^\Sigma$ over $\mathcal{V}$, we get
\begin{align*}
\tr_{\mathcal{V}}q^{\Sigma}\eqdef \sum_{\alpha=1}^{n+1} q^{\Sigma} (V_\alpha^\perp,V_\alpha^\perp)=-k(n-k),
\end{align*}
where $\{V_\alpha\}_{\alpha=1}^{n+1}$ is any orthonormal basis of $\mathcal{V}$. In particular, if we integrate over $\Sigma$, we get
\[
\tr_{\mathcal{V}} {Q}^{\Sigma} \eqdef \sum_{\alpha=1}^{n+1} Q^\Sigma(V_\alpha^\perp, V_\alpha^\perp) =  \int_\Sigma \tr_{{\mathcal V}} q^\Sigma \de\gb = -k(n-k)\vol_g(\Sigma).
\]
\end{theorem}

\begin{proof}
First observe that, given $V\in\mathcal{V}$ induced by the constant vector $v\in\R^{n+1}$, namely $V(x) = v-\scal{v}{x}x$ for all $x\in S^n$, we have
\[
\nabla _X V=(D_X V)^{TS^n} =-(D_X (\langle v,x\rangle x))^{TS^n}=-\langle v,x\rangle (D_X x)^{TS^n}=-\langle v,x\rangle X,
\]
for all vector fields $X\in\vf(S^n)$, where $D$ is the covariant derivative of $\R^{n+1}$.
Therefore, if we assume that $X$ is tangent to $\Sigma$, we get that
\[
\nabla^\perp_X V^\perp=(\nabla_X V -\nabla_X V^\top)^\perp = -A(X,V^\top).
\]

Fixed $p\in\Sigma$ and given a local orthonormal frame $\{E_i\}_{i=1}^k$ of $\Sigma$ around $p$, we can thus deduce that
\begin{align*}
q^{\Sigma}(V^\perp, V^\perp)&=\lvert\nabla_{\Sigma}^\perp V^\perp\rvert^2_{\gb}- \tr_\Sigma( R_{S^n}(V^\perp,\cdot)V^\perp)-
    \lvert \gb(A(\cdot,\cdot),V^\perp) \rvert^2\\
&=\sum_{i=1}^k \abs{A(E_i,V^\top)}^2_\gb-k\abs{V^\perp}^2_\gb-\sum_{i,j=1}^k \abs{\gb( A(E_i,E_j),V^\perp)}^2.
\end{align*}

Now, let $\{{V}_\alpha\}_{\alpha=1}^{n+1}$ be an orthonormal basis of ${\mathcal{V}}$, induced by orthonormal vectors $\{v_\alpha\}_{\alpha=1}^{n+1}$ in $\R^{n+1}$ such that
\begin{itemize}
    \item $v_1={E}_1(p),\ldots,v_k={E}_k(p)$ are tangent to $\Sigma$ at $p$;
    \item $v_{k+1},\ldots,v_n$ are normal to $\Sigma$ and tangent to $S^n$ in $p$;
    \item $v_{n+1}=p$.
\end{itemize}
With this choice, observe that $V_i^\top(p) = E_i(p)$ and $V_i^\perp(p) = 0$ for $i\in\{1,\ldots,k\}$, $V_r^\top(p) = 0$ and $V_r^\perp(p) = v_r$ for $r\in\{k+1,\ldots,n\}$, and $V_{n+1}(p) = 0$.
Therefore, at the point $p$, we get that
\begin{align*}
\tr_{\mathcal{V}} q^\Sigma &=\sum_{\alpha=1}^{n+1} q^{\Sigma}(V_\alpha^\perp, V_\alpha^\perp)=\sum_{\alpha=1}^{n+1}\sum_{i=1}^k \abs{A(E_i,V^\top_\alpha)}^2_\gb-k\sum_{\alpha=1}^{n+1}\abs{V^\perp_\alpha}^2_\gb-\sum_{\alpha=1}^{n+1}\sum_{i,j=1}^k \abs{\gb( A(E_i,E_j),V^\perp_\alpha)}^2\\
&= \sum_{i,j=1}^k \abs{A(E_i,E_j)}^2_\gb - k(n-k) - \sum_{i,j=1}^k \abs{A(E_i,E_j)}^2_\gb = -k(n-k).
\end{align*}
By the independence of the trace from the choice of the orthonormal basis and the arbitrariness we chose $p$, we conclude.
\end{proof}

Given the previous result, it is natural to consider the space $\tilde{\mathcal{V}}=\{\tilde{V}=e^{-f}V \st V\in\mathcal{V}\}$, endowed with the inner product induced by $\mathcal{V}$, as competitors for stability of minimal submanifolds in $(S^n,\gc)$. 

\begin{theorem} \label{thm:TraceSecondVariationCC}
Let $\Sigma^k$ be a k-dimensional closed minimal submanifold of $(S^n,\gc)$, then
\begin{align*}
\tr_{\tilde{\mathcal{V}}} \tilde{q}^{\Sigma}\eqdef \sum_{\alpha=1}^{n+1}\tilde q^\Sigma(\tilde V_\alpha^\perp,\tilde V_\alpha^\perp)&= - \tilde{K}_{S^n} (\Sigma,N\Sigma)+k\abs{\nabla^\perp f}^2_{\gb} e^{-2f}\\
&=- \tilde{K}_{S^n} (\Sigma,N\Sigma)+\frac{1}{k}\abs{H}^2_{\gb} e^{-2f},
\end{align*}
where $\{\tilde V_\alpha\}_{\alpha=1}^{n+1}$ is any orthonormal basis of $\tilde {\mathcal{V}}$. Here, by $\tilde{K}_{S^n} (\Sigma,N\Sigma)$, we mean \[\tilde{K}_{S^n} (\Sigma,N\Sigma)\eqdef\sum_{i=1}^k \sum_{r=k+1}^n \tilde{K}_{S^n} (\tilde{E}_i,\tilde{E}_r),\]
where $\{\tilde{E}_\alpha\}_{\alpha=1}^n$ is an orthonormal basis of $S^n$ with respect to $\gc$ such that the first $k$ terms form an orthonormal basis for $\Sigma$.
\end{theorem}
\begin{remark}
Observe that
\begin{align*}
\tilde K_{S^n}(\Sigma,N\Sigma) & = \sum_{i=1}^k \sum_{r=k+1}^n \tilde{K}_{S^n} (\tilde{E}_i,\tilde{E}_r)= \sum_{i=1}^k \sum_{r=k+1}^n \gc(\tilde{R}_{S^n} (\tilde{E}_i,\tilde{E}_r)\tilde E_i,\tilde E_r).
\end{align*}
Since $\tilde R_{S^n}$ is a tensor, this proves that $\tilde K_{S^n}(\Sigma,N\Sigma)$ does not depend on the choice of basis $\{\tilde E_\alpha\}_{\alpha=1}^n$.
\end{remark}
\begin{proof}
Fix $p\in\Sigma$ and let $\{\tilde{V}_\alpha\}_{\alpha=1}^{n+1}$ be an orthonormal basis of $\tilde{\mathcal{V}}$ such that $V_\alpha\eqdef e^f\tilde V_\alpha\in\mathcal{V}$ is induced by the vector $v_\alpha$ in $\R^{n+1}$ for all $\alpha\in\{1,\ldots,n+1\}$, where $\{v_\alpha\}_{\alpha=1}^{n+1}$ is an orthonormal basis of $\R^{n+1}$. Moreover, assume that
\begin{itemize}
    \item $E_1\eqdef v_1,\ldots,E_k\eqdef v_k$ are tangent to $\Sigma$ at $p$;
    \item $E_{k+1} \eqdef v_{k+1},\ldots,E_n\eqdef v_n$ are normal to $\Sigma$ and tangent to $S^n$ in $p$;
    \item $v_{n+1}=p$.
\end{itemize}
Observe that, in this way, $\{E_\alpha\}_{\alpha=1}^n$ is an orthonormal basis of $T_p S^n$ with respect to $\gb$ such that the first $k$ terms form an orthonormal basis for $\Sigma$.
Moreover, we have $V_i^\top(p) = E_i$ and $V_i^\perp(p) = 0$ for $i\in\{1,\ldots,k\}$, $V_r^\top(p) = 0$ and $V_r^\perp(p) = E_r$ for $r\in\{k+1,\ldots,n\}$, and $V_{n+1}(p) = 0$.

It follows that, at the point $p$, it holds 
\begin{align*}
\sum_{\alpha=1}^{n+1} \tilde{q}^{\Sigma}(\tilde{V}^\perp_{\alpha},\tilde{V}^\perp_{\alpha})e^{2f}&=\sum_{\alpha=1}^{n+1} q^{\Sigma} (V_\alpha^\perp,V_\alpha^\perp)-(n-k)\abs{\nabla^\top f}^2_\gb  + k(n-k) \abs{\nabla f}^2_\gb  +{}\\
&\pheq + (n-k) \div_\Sigma(\nabla f)   + k\sum_{r=k+1}^n \Hess f(E_r,E_r)
\\
&=-k(n-k)-(n-k)\abs{\nabla^\top f}^2_\gb  + k(n-k) \abs{\nabla f}^2_\gb  +{}\\
&\pheq + (n-k) \div_\Sigma(\nabla f)   + k\sum_{r=k+1}^n \Hess f(E_r,E_r),
\end{align*}
where we used \cref{cor:SecondVariationCC} and \cref{thm:LawsonSimonsClosed}.

However, observe that, by \cref{prop:UnderConformalChange}\ref{ucc:sect}, we have
\begin{align*}
\sum_{i=1}^k\sum_{r=k+1}^n \tilde K_{S^n}(\tilde E_i,\tilde E_r)e^{2f} &=\sum_{i=1}^k\sum_{r=k+1}^n \tilde K_{S^n}(E_i,E_r)e^{2f}\\
&=\sum_{i=1}^k\sum_{r=k+1}^n  K_{S^n}(E_i,E_r) + (n-k) \abs{\nabla^\top f}^2_\gb + k \abs{\nabla^\perp f}^2_\gb -k(n-k) \abs{\nabla f}^2_\gb +{}\\
&\pheq - (n-k)\div_\Sigma(\nabla f) - k \sum_{r=k+1}^n \Hess f(E_r,E_r)\\
&= k(n-k) + (n-k) \abs{\nabla^\top f}^2_\gb + k \abs{\nabla^\perp f}^2_\gb -k(n-k) \abs{\nabla f}^2_\gb +{}\\
&\pheq - (n-k)\div_\Sigma(\nabla f) - k \sum_{r=k+1}^n \Hess f(E_r,E_r).
\end{align*}
Therefore we get that
\begin{align*}
\tr_{\tilde{\mathcal V}} \tilde q^\Sigma= \sum_{\alpha=1}^{n+1} \tilde{q}^{\Sigma}(\tilde{V}^\perp_{\alpha},\tilde{V}^\perp_{\alpha})&=-\sum_{i=1}^k\sum_{r=k+1}^n \tilde K_{S^n}(\tilde E_i,\tilde E_r) + k \abs{\nabla^\perp f}^2_\gb e^{-2f}.
\end{align*}
Finally, the last equality in the statement follows from the fact that $k\nabla^\perp f = H$ by \cref{lem:MeanCurvatureCC}, since $\Sigma$ is minimal with respect to $\gc$ and therefore $\tilde H = 0$.
\end{proof}

We are now ready to prove our main result \cref{thm:main}. Note that we will crucially use the assumption that $k\geq 2$. Indeed, we will apply \cref{prop:UnderConformalChange}\ref{ucc:sect} to (orthonormal) vectors tangent to the submanifold, and hence, we need the dimension of $\Sigma$ to be at least two.

\begin{proof} [Proof of \cref{thm:main}]
Let $\Sigma^k$ be a closed $k$-dimensional minimal submanifold of $(S^n,\gc)$ and let $\{{E}_i\}_{i=1}^k$ be a local orthonormal frame of $\Sigma$ with respect to $\gb$. For all $i\not=j\in\{1,\ldots,k\}$, using \cref{prop:UnderConformalChange}\ref{ucc:sect}, we have
\begin{align*}
e^{2f} \tilde{K}_{S^n}({E}_i,{E}_j)=K_{S^n}({E}_i,{E}_j) +E_i(f)^2 + E_j(f)^2 - \abs{\nabla f}^2_{\gb} - \Hess f (E_i,E_i) - \Hess f (E_j,E_j).
\end{align*}
Summing over all $i\not=j\in\{1,\ldots,k\}$, we get
\[
e^{2f} \sum_{i\not=j=1}^k \tilde{K}_{S^n}({E}_i,{E}_j)=k(k-1) +2(k-1)\abs{\nabla^\top f}^2_{\gb} - k(k-1)\abs{\nabla f}^2_{\gb} - 2(k-1) \div_\Sigma (\nabla f).
\]

Multiplying by $e^{(k-2)f}$ and integrating over $\Sigma$ with respect to the metric $\gb$, we obtain
\begin{align*}
\int_\Sigma \sum_{i\not=j=1}^k \tilde{K}_{S^n}({E}_i,{E}_j) \de\gc &= \int_\Sigma \sum_{i\not=j=1}^k \tilde{K}_{S^n}({E}_i,{E}_j) e^{kf} \de \gb \\
&= k(k-1)\int_\Sigma e^{(k-2)f}\de \gb + \int_\Sigma \left[2(k-1)\abs{\nabla^\top f}^2_{\gb} - k(k-1)\abs{\nabla f}^2_{\gb} \right] e^{(k-2)f}\de \gb + {}\\
&\pheq {}- 2(k-1) \int_\Sigma\div_\Sigma (\nabla f) e^{(k-2)f}\de \gb .
\end{align*}
Now observe that, using the ``generalized divergence theorem'' (see e.g. \cite[Theorem 1]{White2016}) and \cref{lem:MeanCurvatureCC}, we can write the last term as
\begin{align*}
\int_\Sigma\div_\Sigma (\nabla f) e^{(k-2)f}\de \gb &= \int_\Sigma \div_\Sigma(e^{(k-2)f}\nabla f) \de \gb - (k-2)\int_\Sigma \abs{\nabla^\top f}^2_{\gb}e^{(k-2)f} \de \gb\\
&=-\int_\Sigma \gb(H_{\gb},\nabla f) e^{(k-2)f} \de \gb - (k-2)\int_\Sigma \abs{\nabla^\top f}^2_{\gb}e^{(k-2)f} \de \gb\\
&=-k\int_\Sigma \abs{\nabla^\perp f}_{\gb}^2 e^{(k-2)f} \de \gb - (k-2)\int_\Sigma \abs{\nabla^\top f}^2_{\gb}e^{(k-2)f} \de \gb.
\end{align*}
Hence, we get
\begin{align*}
\int_\Sigma \sum_{i\not=j=1}^k \tilde{K}_{S^n}({E}_i,{E}_j) \de\gc&= k(k-1)\int_\Sigma e^{(k-2)f}\de \gb + (k-1)(k-2)\int_\Sigma \abs{\nabla^\top f}^2_{\gb} e^{(k-2)f}\de \gb+{}\\
&\pheq {}+  k(k-1)\int_\Sigma\abs{\nabla^\perp f}^2_{\gb}  e^{(k-2)f}\de \gb\\
& > k(k-1)\int_\Sigma\abs{\nabla^\perp f}^2_{\gb}  e^{(k-2)f}\de \gb.
\end{align*}

Combining with \cref{thm:TraceSecondVariationCC} we obtain 
\begin{align*}
\tr_{\tilde{\mathcal{V}}} \tilde{Q}^{\Sigma} =\int_\Sigma \tr_{\tilde{\mathcal V}} \tilde q^\Sigma \de\gc&= \int_{\Sigma} - \sum_{i=1}^k\sum_{r=k+1}^n\tilde{K}_{S^n} (E_i,E_r)+k\abs{\nabla^\perp f}^2_{\gb} e^{-2f} d\gc \\
&< \int_{\Sigma} - \sum_{i=1}^k\sum_{r=k+1}^n\tilde{K}_{S^n} (E_i,E_r)+ \frac{1}{(k-1)}\sum_{i\not=j=1}^k \tilde{K}_{S^n}({E}_i,{E}_j) \de \gc\\
&\le \int_{\Sigma} - k(n-k)\min_{\pi\subset T_p M}\tilde{K}_{S^n}(\pi) + k \max_{\pi\subset T_pM}\tilde{K}_{S^n}(\pi) \de \gc(p) \\
&\le -k\int_\Sigma \min_{\pi\subset T_p M}\tilde{K}_{S^n}(\pi) (n-k-\delta^{-1})\de \gc(p)< 0,
\end{align*}
where we used that $k\le n-\delta^{-1}$ and that $(M,\gc)$ is $\delta$-pinched.
\end{proof}

\bibliography{biblio}

\printaddress
\end{document}